\documentclass[reqno]{amsart}
\usepackage{tikz}
\usetikzlibrary{patterns} 
\usepackage{pgfplots}
\usepackage{amsfonts,amsmath,amsthm,amssymb,mathrsfs}
\usepackage{color}
\usepackage{amsmath,amssymb,amsfonts, bm}
\usepackage{graphicx}
\usepackage{newunicodechar}
\usepackage{xcolor}
\usepackage{geometry}
\numberwithin{equation}{section}
\usepackage[pdfstartview=FitH,colorlinks=true]{hyperref}
\hypersetup{urlcolor=red, linkcolor=blue,citecolor=red}
\allowdisplaybreaks

\theoremstyle{plain}

\newtheorem{thm}{Theorem}[section]
\newtheorem{definition}[thm]{Definition}
\newtheorem{prop}[thm]{Proposition}
\newtheorem{cor}[thm]{Corollary}
\newtheorem{lemma}[thm]{Lemma}

 \numberwithin{equation}{section}

\begin{document}
\title[]
{Dispersive decay for the Inter-critical nonlinear Schr\"{o}dinger equation in $\mathbb{R}^3$}

\author[]
{Boyu Jiang, Jiawei Shen, Kexue Li*}

\address{Boyu Jiang
 \newline\indent
 School of Mathematics and Statistics, Xi'an Jiaotong University, Xi'an ,710049, Shanxi,China
 \newline\indent
 Jiawei Shen
 \newline\indent
 School of Mathematics and Computational Science, Xiangtan
University, Xiangtan, 411105, Hunan, China
 \newline\indent
 Kexue Li
 \newline\indent
 School of Mathematics and Statistics, Xi'an Jiaotong University,Xi'an 710049, Shanxi,China
 \newline\indent
 }
 \email{ jiangboyu@stu.xjtu.edu.cn; sjiawei633@xtu.edu.cn; kxli@mail.xjtu.edu.cn}

\begin{abstract}
This paper investigates the Cauchy problem for the nonlinear Schrödinger equation (NLS) in the mass-supercritical and energy-subcritical regime within three spatial dimensions. For initial data in the critical homogeneous Sobolev space $\dot{H}^{s_c}(\mathbb{R}^3)$ (where $s_c = \frac{5}{6}$), we get a uniform decay estimate for the long-time dynamics of solutions, which extends the previous results.
\end{abstract}

\maketitle

\section{Introduction}
In this paper, we primarily discuss the $H^{s_c}$-critical nonlinear Schr\"{o}dinger equation (NLS). The general form of the equation is as follows, where $s_c=\frac{d}{2}-\frac{2}{p}$. The equation takes the form:
\begin{equation} \label{NLS}
i u_t + \Delta u \pm |u|^{p} u = 0, \
u(0, x) = u_0(x) \in \dot{H}^{s_c}(\mathbb{R}^{d}).
\end{equation}
When the sign of the nonlinear term is positive, it is referred to as the focusing case, and when it is negative, it is referred to as the defocusing case. The solution of equation \ref{NLS} remains invariant under the scaling
\begin{equation}
u(t,x) \to u_\lambda(t,x) = \lambda^{\frac{2}{p}} u(\lambda^2 t, \lambda x) \quad \text{for } \lambda > 0,
\end{equation}
and the initial value transforms into
\begin{equation}
u(0) \to u_\lambda(0) := \lambda^{\frac{2}{p}} u_0(\lambda x) \quad \text{for } \lambda > 0.
\end{equation}
Then the scaling leaves $\dot{H}^{s_c}$ norm invariant, that is
$$\| u \|_{\dot{H}^{s_c}} = \| u_\lambda \|_{\dot{H}^{s_c}},$$
which is called the critical regularity $s_c$. This is also considered to be the minimal regularity of the initial data required to guarantee the well-posedness of equation \eqref{NLS}. In fact, Christ, Colliander, and Tao\cite{ill-posedness} have already proved that there exist some initial data belonging to $H^s(\mathbb{R}^d)$ with $s < s_c$ that cause \eqref{NLS} to be ill-posed. \par
The $H^1$ solutions of equation \ref{NLS} satisfy the conservation laws of mass, momentum, and energy, which read
\begin{align*}
M(u(t)) &:= \int |u(t,x)|^2 dx = M(u_0), \\
P(u(t)) &:= \operatorname{Im} \int \overline{u(t,x)} \nabla u(t,x) \, dx = P(u_0), \\
E(u(t)) &:= \int |\nabla u(t,x)|^2 dx + \frac{2\mu}{p+2} \int |u(t,x)|^{p+2} dx = E(u_0).
\end{align*}
For the Cauchy problem of equation \eqref{NLS}, the well-posedness and scattering theory (with initial data in $H^s(\mathbb{R}^n)$) have been extensively studied. Local well-posedness can be derived from standard fixed-point theorems: for all $u_0 \in H^s(\mathbb{R}^d)$, there exists a $T_0 > 0$ such that its solution $u \in C([0,T_0), H^s(\mathbb{R}^d))$ (strong solution). In fact, this $T_0$ depends on $\|u_0\|_{H^s(\mathbb{R}^d)}$ when $s > s_c$. When $s = s_c$, it also depends on the profile of $u_0$. Some results can be found in Cazenave and Weissler \cite{CAZENAVE1990807}.\par
These techniques can be directly applied to prove the global well-posedness of equation \eqref{NLS} with small initial data in $H^s(\mathbb{R}^d)$ ($s \ge s_c$). The main theorem is stated as follows, see \cite{Miao2023}
\begin{thm}(Well-posedness) \label{well-posedness}
    Let $d \ge 1$, and $u_0(x) \in H^{s_c}$. Furthermore, assume $0 \le s_c \le 1$. If $\||\nabla|^{s_c}u_0\|_{L^2} < \delta(d)$ is sufficiently small, then $u(t)$ does not blow up forward or backward in time (global existence holds). That is, scattering holds, and
    \begin{equation}  \label{origin}
        \||\nabla|^{s_c}u\|_{L^q_t(\mathbb{R}) L^r_x(\mathbb{R}^{d})} \lesssim
        \||\nabla|^{s_c}u_0\|_{L^2(\mathbb{R}^d)},
    \end{equation}
    where $(q, r)$ is a Schrödinger admissible pair.
\end{thm}
In the mass-supercritical and energy subcritical case $\frac{4}{d}<p<\frac{4}{d-2}$, when the solution is considered in the energy space $H^1(\mathbb{R}^d),$  we can use local theory together with conservation to yield global well-posedness for any initial data $u_0\in H^1(\mathbb{R}^d)$ in the defocusing case and some special focusing cases.  Furthermore, scattering theory under the same conditions was studied by Ginibre and Velo \cite{Duyckaerts2008,Ginibre1985}.  However, in the mass-critical and energy-critical case, it is not easy to obtain global well-posedness because conservation does not imply the global existence of the solutions.  The problem was firstly considered by Bourgain \cite{Bourgain1999}, and he considered the radial function under the energy-critical and defocusing case.  After that, many author studied the case, and please read \cite{Keel1998,Ryckman2007,Visan2007,Visan2012}. On the other hand,  the focusing case was also considered by Kenige and Merle\cite{Kenig2006} when the initial data is a radial function, and about non-radial case, please refer to \cite{Killip2011,Dodson2012,Killip2010}. In the mass-critial, Killip, Tao, Visan \cite{Killip2008} first  studied the global well-posedeness and scattering with radial data in dimension two, and  Killip, Visan, Zhang\cite{Killip2008}  consider it in dimension higher than two. About non-radial case, please read a series of papers of Dodson\cite{Dodson2012,Dodson2015,Dodson2016}.\par
When considering the general nonlinear Schrödinger equation in the critical space $\dot{H}^{s_c}(\mathbb{R}^d)$, more complex situations arise. Recently, results on conditional global existence and scattering under the assumption that $u \in L_t^\infty(I, \dot{H}_x^{s_c}(\mathbb{R}^d))$ (where $I$ denotes the maximal lifespan interval) have also been studied by many authors. This topic of research began with the papers \cite{EKenig2007ScatteringFH, Nondispersive} and was subsequently refined in works such as \cite{dimensionfive, spacedimensions, Scattering, Blow-up, Radialsolutionstoenergysupercriticalwaveequationsinodddimensions, Largedata, Theradialmass-subcritical, Energy-supercritical, Theradialdefocusing, Thefocusingenergy-critical, Thedefocusingenergy-supercritical, IntercriticalNLS, Theradialdefocusingnonlinear, Globalwell-posednessandscattering}. Specifically, if the initial data $u_0 \in \dot{H}^{s_c}(\mathbb{R}^d)$ and the solution satisfies an a priori estimate
\begin{equation}
    \sup_{0 < t < T_{\mathrm{out}}(u_0)} \| u \|_{\dot{H}_x^{s_c}(\mathbb{R}^d)} < +\infty,
\end{equation}
then $T_{\mathrm{out}}(u_0) = +\infty$, and the solution scatters in $\dot{H}^{s_c}(\mathbb{R}^d)$, where $[0, T_{\mathrm{out}}(u_0))$ is the maximal forward lifespan of the solution. Hence, these results establish a blow-up criterion where the lifespan depends only on the critical norm $\| u \|_{L_t^\infty \dot{H}_x^{s_c} (I \times \mathbb{R}^d)}$. However, we have limited understanding regarding large initial data problems if we only require $u_0 \in \dot{H}^{s_c}(\mathbb{R}^d)$. Subsequently, many scholars have approached global solutions with large initial data from a probabilistic perspective. Specifically, they construct a broad class of initial data sets with supercritical regularity such that global solutions exist. For more details, refer to\cite{Periodicnonlinear,Invariantmeasures,Randomdata,RandomdataCauchy,Probabilisticwell-posedness,Almostsurewell-posedness,Two-dimensionalnonlinear,Invariantmeasuresandlongtime,Almostsurescattering,RandomdataCauchytheory,Invariantweighted,Almostsurewell-posednessfortheperiodic,Ontheprobabilisticwell-posedness,Probabilisticglobalwell-posednessoftheenergy-critical,Almostsureglobalwell-posedness,Probabilisticglobalwell-posedness,RandomdataCauchyproblemforsupercritical}. \par
In our article, we consider the nonlinear Schrödinger equation in the mass-supercritical and energy-subcritical regime (inter-critical), with initial data belonging to the critical space $\dot{H}^{s_c}(\mathbb{R}^d)$. We study the long-time behavior of solutions to the following equation:
\begin{equation} \label{eq}
\begin{cases} 
i \partial_t u + \Delta u \pm |u|^{3} u=0 \\
u(0, x) = u_0(x) \in \dot{H}^{\frac{5}{6}{}}(\mathbb{R}^d).
\end{cases}
\end{equation}
Prior to the advent of Strichartz estimates, our understanding of the long-time behavior of solutions relied on a different expression of dispersion, namely, quantitative decay pointwise in time. For the linear Schrödinger equation, this is exemplified by the classical dispersive
estimate
\begin{equation} \label{dispersive}
    \| e^{it\Delta} u_0 \|_{L_x^p(\mathbb{R}^d)} \lesssim |t|^{-d\left( \frac{1}{2} - \frac{1}{r} \right)} \| u_0 \|_{L_x^{p'}} \quad \text{for } t \neq 0 \text{ and } 2 \leq p \leq \infty.
\end{equation}
The requirement that $u_0 \in L^{p'}$ constitutes a spatial concentration requirement for the initial data. This is what breaks the time translation symmetry and so makes a quantitative decay estimate possible.
Our main theorem is as follows:
\begin{thm} \label{main thm}
    Fix $2 < p < \infty$. Given $u_0 \in \dot{H}^{\frac{5}{6}}(\mathbb{R}^3) \cap L^{p'}(\mathbb{R}^3)$, we denote by $u$ the global solution generated by Theorem \ref{well-posedness}. Then we have the following estimate:
    \begin{equation} \label{main eq}
        \sup_{t \neq 0} |t|^{3\left( \frac{1}{2} - \frac{1}{p} \right)} \| u(t) \|_{L_x^p} \leq C \left( \| u_0 \|_{\dot{H}_x^{\frac{5}{6}}} \right) \| u_0 \|_{L_x^{p'}}.
    \end{equation}
\end{thm}

\subsection*{Notation}  Throughout this paper, $C$ is a postive finite constant which is independent of  the essential variables. $ A\lesssim B$ mean $A\leq CB$ for some constant. If both $A\lesssim B$ and $B\lesssim A,$ then we donate $A\sim B.$  We will use $ X\hookrightarrow Y$ to mean a continuous embedding, that, for two function spaces $X,Y,$  there an inclusion map $X\rightarrow Y$ with $\|f\|_Y\lesssim\|f\|_X.$

Our conventions for the Fourier transform are \[
\hat{f}(\xi)=\frac{1}{\sqrt{2\pi}}\int e^{-i\xi\cdot x}f(x)dx \text{ and } f(x)=\frac{1}{\sqrt{2\pi}}\int e^{-i\xi\cdot x}\hat{f}(\xi)d\xi. 
\]
We get that Fourier transform is  a unitary on $L^2$ with $\|\hat{f}\|_{L^2}=\|f\|_{L^2}.$ If $f(x,t)$ is defined on both space and time, we use $\hat{f}(\xi,t)$ to denote the fourier transform on spatial variable.

\section{some setting and bascial Lemma}
In this section, we will introduce two kind of function space:    
\subsection{Homogeneous Sobolev Space}
   Now, we begin to introduce the homogeneous Sobolev space $\dot{H}^s$ definded by 
   \[
   \dot{H}^s:=\{f \text{ is temped distribution and  } \|f\|_{\dot{H}^s} \le \infty\}
   \]
   where $\|f\|_{\dot{H}^s}=\int |\xi|^{2s} |\hat{f}(\xi)d\xi.$ Now, in order to be convenience, we introduce the Litterwood-Paley decomposition to study $\dot{H}^s.$ Let $\phi$ be smooth function supported in $B(0,2)$ and $\phi\equiv 1$ when $\xi\le 1.$  So, for dyadic $2^N\in 2^\mathbb{Z},$  we then define $P_{\le N}, P_{N}$ and  $P_{> N}$

   \begin{align*}
      & \widehat{P_{\le N} f}(\xi)=\phi(\xi/N)\hat{f}(\xi)\\
        &\widehat{P_{ N} f}(\xi)=(\phi(\xi/N)-\phi(2\xi/N))\hat{f}f(\xi)\\
         &\widehat{P_{ >N} f}(\xi)=(1-\phi(\xi/N))\hat{f}(\xi)
   \end{align*}
To be simple, we can also write $f_{\le N}=P_{\le N}, f_{N}= P_{N}f$ and $f_{>N}=P_{ >N} f.$   It is easy to check that
\[
\|\|f_N\|_{\ell^2_N}\|_{L^p}\sim \|f\|_{L^p}
\]
for $1<p<\infty.$  For a function $f(x,t)$ defined on both space and time, we also use $ P_{N}f(x,t)$ to denote Littlewood-Paley projection of $f$ in only the spatial variable $x.$ 
\begin{thm}[Bernstein inequality]\label{Bernstein_inequality}
Let $1\le q\le p\le \infty,$  and $P_{\le N}, P_{N}$ is Littlewood-Paley operator. For any $s\ge0,$ we have 
\begin{equation}\begin{split}\label{intro/bernstein}
    \||\nabla|^s P_N f\|_{L^p} & \approx N^s \|P_N f\|_{L^p} \\
    \||\nabla|^s P_{\leq N} f\|_{L^p} & \lesssim N^s \|P_N f\|_{L^p} \\
    \|P_N f\|_{L^p},\|P_{\leq N} f\|_{L^p} & \lesssim N^{d(\frac{1}{q} - \frac{1}{p})} \|P_N f\|_{L^q}.
\end{split}\end{equation}

\end{thm}

\subsection{ Lorentz Space}

In this section, we will offer some basic knowledge about Lorentz space.  About some detail, please read \cite{Grafakos2014}
\begin{definition}[Lorentz space] \label{Lorentz_Space}
Fix \( d \geq 1 \), \( 1 \leq p < \infty \), and \( 0 < q \leq \infty \). The Lorentz space \( L^{p,q} \)
is the space of measurable functions \( f: \mathbb{R}^d \to \mathbb{C} \) which have finite quasinorm
\begin{align}
    \| f \|_{L^{p,q}(\mathbb{R}^d)} = \begin{cases}
\left( p^{1/q} \left\| \lambda \left| \{ x \in \mathbb{R}^d : |f(x)| > \lambda \} \right|^{1/p} \right\|_{L^q((0,\infty), \frac{d\lambda}{\lambda})} \right)
\end{cases},
\end{align}

where \( |\ast| \) denotes the Lebesgue measure on \( \mathbb{R}^d \).

\end{definition}

It follows that \( L^{p,q} \) is a quasi-Banach space for any \( 1 \leq p < \infty \) and \( 0 < q \leq \infty \).  Furthermore,  \( 1 < p < \infty \) and \( 1 \leq q \leq \infty \), we find that
\[
\| f \|_{L^{p,q}} \sim_{p,q} \sup_{\| g \|_{L^{p',q'}} = 1} \left| \int f(x) \overline{g(x)} dx \right|
\]
where \( p', q' \) are the respective H\"{o}lder conjugates. Therefore for all \( 1 < p < \infty \) and \( 1 \leq q \leq \infty \), it follows that \( L^{p,q} \) is normable. In the case of \( p = q \), \( L^{p,p}(\mathbb{R}^d) \) coincides with the standard Lebesgue space \( L^p(\mathbb{R}^d) \). By the Definition \ref{Lorentz_Space},  it is easy to check that \[
\bigl\| |x|^{-d/p} \bigr\|_{L^{p,\infty}(\mathbb{R}^d)} \sim_d 1, \tag{2.2}
\]
and hence \( |x|^{-d/p} \in L^{p,\infty}(\mathbb{R}^d) \) for all \( p \geq 1 \). This is the extent to which we will use the exact form of (2.1).

In the same manner as the sequence spaces \( \ell^q \), the Lorentz spaces \( L^{p,q} \) satisfy a nesting property in the second index \( q \). In particular, we have the continuous embedding \( L^{p,q_1} \hookrightarrow L^{p,q_2} \), i.e.
\[
\| * \|_{L^{p,q_2}} \lesssim_{p,q_1,q_2} \| * \|_{L^{p,q_1}},
\]
for all \( 0 < q_1 \leq q_2 \leq \infty \).

Lorentz spaces arise most naturally as real interpolation spaces between the usual \( L^p \) spaces. This is achieved through the Hunt interpolation inequality, otherwise known as the off-diagonal Marcinkiewicz interpolation theorem; see \cite{Grafakos2014}. We recall a specific case of the theorem here:
    
\begin{lemma}[Hunt interpolation] \label{Hunt_interpolation}
Fix \( 1 \leq p_1, p_2, q_1, q_2 \leq \infty \) such that \( p_1 \neq p_2 \) and \( q_1 \neq q_2 \). Let \( T \) be a sublinear operator which satisfies
\[
\| Tf \|_{L^{p_i}} \lesssim_{p_i,q_i} \| f \|_{L^{q_i}}
\]
for \( i \in \{1,2\} \). Then for all \( \theta \in (0,1) \) and all \( 0 < r \leq \infty \),
\[
\| Tf \|_{L^{p_\theta,r}} \lesssim_{p_0,q_0,\theta} \| f \|_{L^{q_\theta,r}}
\]
where \( \frac{1}{p_\theta} = \frac{\theta}{p_1} + \frac{1-\theta}{p_2} \) and \( \frac{1}{q_\theta} = \frac{\theta}{q_1} + \frac{1-\theta}{q_2} \).
\end{lemma}
Lorentz spaces enjoy many of the standard estimates used in the Lebesgue spaces \( L^p \). In particular, H\"older's inequality carries over in the following form:

\begin{lemma}[H\"older's inequality] \label{H_inequality}
Given \( 1 \leq p, p_1, p_2 \leq \infty \) and \( 0 < q, q_1, q_2 \leq \infty \) such that \( \frac{1}{p} = \frac{1}{p_1} + \frac{1}{p_2} \) and \( \frac{1}{q} = \frac{1}{q_1} + \frac{1}{q_2} \),
\[
\| fg \|_{L^{p,q}} \lesssim_{d,p,q} \| f \|_{L^{p_1,q_1}} \| g \|_{L^{p_2,q_2}}.
\]
\end{lemma}

In addition, Lorentz spaces satisfy the Young--O'Neil convolutional inequality, see \cite{Blozinski1972,Nakanishi2001,ONeil1963}, of which the Hardy--Littlewood--Sobolev inequality is a special case:
 
\begin{lemma}[Young--O'Neil convolutional inequality] \label{ convolutional_inequality}
Given \( 1 < p, p_1, p_2 < \infty \) and \( 0 < q, q_1, q_2 \leq \infty \) such that \( \frac{1}{p} + 1 = \frac{1}{p_1} + \frac{1}{p_2} \) and \( \frac{1}{q} = \frac{1}{q_1} + \frac{1}{q_2} \),
\[
\| f * g \|_{L^{p,q}} \lesssim_{d,p_i,q_i} \| f \|_{L^{p_1,q_1}} \| g \|_{L^{p_2,q_2}}.
\]
\end{lemma}

From Hunt interpolation and the usual Sobolev embedding theorems, we also find an analog of Sobolev embedding in Lorentz spaces,

\begin{lemma}[Sobolev embedding] \label{Sobolev_embedding}
Fix \( 1 < p < \infty \), \( s \geq 0 \), and \( 0 < \theta \leq \infty \) such that \( \frac{1}{p} + \frac{s}{d} = \frac{1}{q} \). Then
\[
\| f \|_{L^{p,\theta}(\mathbb{R}^d)} \lesssim_{p,s,\theta} \| \nabla^s f \|_{L^{q,\theta}(\mathbb{R}^d)}.
\]
\end{lemma}

Finally, we may show a basic Leibniz rule in Lorentz spaces. We recall that the Schwartz functions \( \mathcal{S}(\mathbb{R}^d) \) are dense in \( L^{p,q} \) for \( q \neq \infty \), see \cite{Grafakos2014}. By the classical Leibniz rule and extending by density, we then find the following lemma:

\begin{lemma}[Leibniz rule] \label{Leibni_rule}
Given \( 1 \leq p, p_i < \infty \), s>0, and \( 0 < q, q_i < \infty \) for \( i \in \{1, 2, 3, 4\} \) such that \( \frac{1}{p} = \frac{1}{p_1} + \frac{1}{p_2} = \frac{1}{p_3} + \frac{1}{p_4} \) and \( \frac{1}{q} = \frac{1}{q_1} + \frac{1}{q_2} = \frac{1}{q_3} + \frac{1}{q_4} \),
\[
\| \nabla^{s}[fg] \|_{L^{p,q}} \lesssim_{d,p,q,p_i,q_i} \| \nabla^{s} f \|_{L^{p_1,q_1}} \| g \|_{L^{p_2,q_2}} + \| f \|_{L^{p_3,q_3}} \| \nabla^{s} g \|_{L^{p_4,q_4}}.
\]
\end{lemma}

\subsection{Lorentz–Strichartz estimates}

\begin{definition}[Schr\"odinger-admissible] \label{admissible}
Fix a spatial dimension \( d \geq 3 \). We say that a pair \( 2 \leq p, q \leq \infty \) is Schr\"odinger-admissible if
\[
\frac{2}{p} + \frac{d}{q} = \frac{d}{2}.
\]
We say that \( (p,q) \) is a non-endpoint Schr\"odinger-admissible pair if \( 2 < p, q < \infty \). Finally, we say that \( (p,q) \) is Schr\"odinger-admissible with \( s \) spatial derivatives if
\[
\frac{2}{p} + d\left( \frac{1}{q} + \frac{s}{d} \right) = \frac{d}{2}.
\]
\end{definition}

\begin{prop}[Lorentz-Strichartz estimates] \label{Strichartz_estimates}
Suppose that \( 2 < p, q < \infty \) is Schr\"odinger-admissible. Then for all \( f \in L^2 \) and any spacetime slab \( J \times \mathbb{R}^d \), the linear evolution satisfies
\[
\bigl\| e^{it\Delta} f \bigr\|_{L_t^p L_x^{q,\theta}(J)} \lesssim_{p,q} \| f \|_{L^2(\mathbb{R}^d)}. \tag{2.3}
\]
Moreover, for all \( 0 < \theta \leq \infty \); \( 1 \leq \phi \leq \infty \); and any time-dependent interval \( I(t) \subset J \),
\[
\left\| \int_{I(t)} e^{i(t-s)\Delta} F(s,x) ds \right\|_{L_t^p L_x^{q,\theta,\phi}(J)} \lesssim_{p,q,\theta,\phi} \| F \|_{L_t^{p'} L_x^{q',\theta',\phi'}(\mathbb{R} \times \mathbb{R}^d)}. \tag{2.4}
\]
\end{prop}

\subsection{Lorentz spacetime bounds}

We may now prove global bounds in mixed Lorentz spacetime norms for solutions to (). We present the proof for all spatial dimensions \( d \geq 3 \) and all non-endpoint Schr\"odinger-admissible pairs.

\begin{prop}[Spacetime bounds] \label{Spacetime_bounds}
Fix \( d \geq 3 \) and \( \phi, \theta \geq 2 \). Suppose that \( 2 < p, q < \infty \) is a Schr\"odinger-admissible pair and suppose that \( u_0 \in \dot{H}^{\frac{5}{6}}(\mathbb{R}^d) \) satisfies the hypotheses of Theorem 1.1. Then the corresponding global solution \( u(t) \) to (NLS) with initial data \( u_0 \) satisfies
\[
\| \nabla^{\frac{5}{6}} u \|_{L_t^{p,\theta} L_x^{q,\phi}} \leq C( \| u_0 \|_{\dot{H}^{\frac{5}{6}}} ).
\]
\end{prop}
\begin{proof}
    At the beginning of the proof, we first recall the classical Strichartz estimate and we have 
  \[
\|\nabla^{\frac{5}{6}} u \|_{L_t^{p} L_x^{q}} \leq C( \| u_0 \|_{\dot{H}^{\frac{5}{6}}} ).
\]
where $(p,q)$ is any Schr\"odinger-admissible pair. Now,  we turn into  the Lorentz Space, and we can write $u$ as 
\[
u(t) = e^{it\Delta} u_0\mp i \int_{0}^{t} e^{i(t-s)\Delta} \left[ |u|^{3} u \right](s) ds
\]
by the Duhamel formula. Furthermore, By Proposition \ref{Strichartz_estimates}, it is easy to get the follow estimate
\[
\| \nabla^{\frac{5}{6}} u \|_{L_t^p L_x^{q,\theta,\phi}} \lesssim \| \nabla^{\frac{5}{6}} u_0 \|_{L^2} + \left\| \nabla^{\frac{5}{6}} |u|^{3} u \right\|_{L_t^{p',\theta'} L_x^{q',\phi'}}.
\]
where $p'\le \theta$ and $q'\le \phi.$ 
Furthermore, 
\begin{equation*}
\| \nabla^{\frac{5}{6}} u \|_{L_t^p L_x^{q,\theta,\phi}}  \lesssim   \| \nabla^{\frac{5}{6}} u_0 \|_{L^2} +\left\| \nabla^{\frac{5}{6}} |u|^{3} u \right\|_{L_t^{p'} L_x^{q'}}\\
\lesssim\| \nabla^{\frac{5}{6}} u_0 \|_{L^2} +\|\nabla^{\frac{5}{6}} u \|_{L_t^{p} L_x^{q}}\left\|u\right\|^{3}_{L_t^{\frac{3p}{p-2}} L_x^{\frac{3q}{q-2}}}
\end{equation*}

A quick calculation shows that $(\frac{3p}{p-2},\frac{3q}{q-2})$ is a non-endpoint Schr\"{o}dinger-admissible pair with $\frac{5}{6}$ spatial derivative. With, this concludes the proof of the proposition for the initial-value problem.
\end{proof}

An unfortunate weakness of Proposition \ref{Spacetime_bounds} is the inability to control Lorentz exponents below 2, which will be necessary in the proof of Lemma 3.3; see (3.13). Though this level of control appears inaccessible for the linear evolution, with (2.4) we can gain additional control over the nonlinear correction.

\begin{cor}[Nonlinear correction bounds] \label{Nonlinear_correction_bounds}
Fix \( d \geq 3 \), \( \theta \geq \frac{2d-2}{d+2} \), and \( \phi \geq \frac{2d-2}{d+2} \vee 1 \). Suppose that \( 2 < p, q < \infty \) is a Schr\"odinger-admissible pair, and suppose that \( u_0 \in \dot{H}^{\frac{5}{6}}(\mathbb{R}^d) \) satisfies the hypotheses of Theorem 1.1. Then the corresponding global solution \( u(t) \) to (NLS) with initial data \( u_0 \) satisfies
\[
\left\| \nabla^{\frac{5}{6}} \int_{0}^{t} e^{i(t-s)\Delta} \left[ |u|^{3} u \right](s) ds \right\|_{L_t^p L_x^{q,\theta,\phi}(\mathbb{R} \times \mathbb{R}^d)} \leq C( \| u_0 \|_{\dot{H}^{\frac{5}{6}}} ).
\]
\end{cor}
\begin{proof}
  We focus on the initial-value problem.

Applying the Strichartz estimate and the nesting of Lorentz spaces, we find
\begin{align*}
    &\left\| \nabla^{\frac{5}{6}} \int_0^t e^{i(t-s)\Delta} \left[|u|^{3} u  \right](s) ds \right\|_{L_t^\rho L_x^{q, \theta} (\mathbb{R} \times \mathbb{R}^d)}\\
    & \lesssim \left\| \nabla^{\frac{5}{6}} \left( u |u|^{3} \right) \right\|_{L_t^{\frac{p}{p-1},\theta} L_x^{\frac{q}{q-1},\phi}}\\
       &\lesssim \| \nabla^{\frac{5}{6}} u \|_{L_t^{p,\frac{\theta(d+2)}{d-1}} L_x^{q,\frac{\phi(d+2)}{d-1}}} \left\| u \right\|^{3}_{L_t^{\frac{3p}{(p-2)},\theta(d+2)} L_x^{\frac{3q}{q-2},\phi(d+2)}} \\
         &\lesssim \| \nabla^{\frac{5}{6}} u \|_{L_t^{p,2} L_x^{q,2}} \| u \|_{L_t^{\frac{3p}{(p-2)},2} L_x^{\frac{3q}{q-2},2} }.
\end{align*}
  It need to point out that $(\frac{3p}{p-2},\frac{3q}{q-2})$ is a non-endpoint Schr\"{o} dinger-admissible pair with $\frac{5}{6}$ spatial derivative. The proof is completed.
\end{proof}

\section{Proof of Theorem \ref{main thm}}
We now commence the proof of the main theorem. By \textit{time-reversal symmetry}, it suffices to prove the dispersive estimate for time $t \in (0, \infty)$. Using Schwartz functions in $\dot{H}^{\frac{5}{6}}(\mathbb{R}^3) \cap L^{p'}(\mathbb{R}^3)$, we only need to show that equation \eqref{main eq} holds for Schwartz solutions of \eqref{eq}. For $T \in (0, \infty]$, we define
\begin{equation*}
    \|u\|_{X(T)} := \sup_{t \in [0, T)} t^{3\left(\frac{1}{2} - \frac{1}{p}\right)} \|u\|_{L^{p}_x}.
\end{equation*}
We employ a small parameter $0<\eta<1$ that will be chosen later depending only on absolut constants(such as those in the dispersive and Strichartz estimates) and $||u_0||_{\dot{H}^{\frac{5}{6}}_x}$. Thanks to \ref{Spacetime_bounds}, We can decompose$[0,\infty)$ into $J=J(||u_0||_{\dot{H}^{\frac{5}{6}}_x},\eta)$ many intervals $I_j=[T_{j-1},T_j)$ such that 
\begin{equation} \label{boot}
    \| u \|_{L_s^{\frac{6p}{6-p}, 3} L_x^{\frac{3p}{p-2}} (I_j) }< \eta.
\end{equation}
We will show that for each $1\le j\le J$, we have 
\begin{equation} \label{1}
    X(T_j)\lesssim ||u_0||_{L^{p'}_x}+C||u_0||_{\dot{H}^{\frac{5}{6}}_x}X(T_{j-1})+\eta^{3}X(T_j).
\end{equation}
 Choosing $\eta$ sufficiently small to defeat the absolute implicit constant in \eqref{1} and $C(||u_0||_{\dot{H}^{\frac{5}{6}}_x})$, we readily obtain 
$$X(\infty)\le C(||u_0||_{\dot{H}^{\frac{5}{6}}_x})||u_0||_{L^{p'}_x}$$
This is precisely the estimate \eqref{main eq} that we set out to prove.
\par
 Now we decompose the proof of Theorem \ref{main thm} into the cases \( 2 < p \le 6 \) and \( 6 < p \). Following a similar approach as in \cite{Dispersivedecay}, for \( 2 < p \le 6 \), we observe that the linear dispersive decay \eqref{dispersive} is integrable near \( t = 0 \). Adopting the notation from \cite{Dispersivedecay}, we refer to \( 2 < p \le 6 \) as the \textit{integrable case} of Theorem \ref{main thm}. This integrability leads to a simplified argument that parallels the proof in \cite{Dispersivedecayforthemass-critical} for the mass-critical nonlinear Schrödinger equation. We present this proof in Subsection 3.1.

For \( p > 6 \), the linear dispersive decay \eqref{dispersive} is no longer integrable near \( t = 0 \), and thus a more nuanced argument is required. The method we employ also originates from \cite{Dispersivedecay}. This will be completed in Subsection 3.2.
 \subsection{Integrable case}
 In this section, we prove Theorem \ref{main thm} in the case where $2<p\le6$. In addition, the result in this section will be used in the proofs of Lemma \ref{4}.\par
 \begin{proof}
    We therefore focus on \eqref{1}. Fix \( t \in [0, T_j) \) and recall the Duhamel formula (variation of parameters formula):

\[
u(t) = e^{i t \Delta} u_0 \mp i \int_{0}^{t} e^{i(t-s)\Delta} \big[ |u|^{3} u \big](s) \, ds.
\]
By the linear dispersive decay \eqref{dispersive}, for the linear term we immediately obtain:
\begin{equation} \label{2}
\big\| e^{i t \Delta} u_0 \big\|_{X(T_j)} \lesssim \| u_0 \|_{L^{p'}_x}.
\end{equation}
Thus we only need to focus on the nonlinear term.
Using the linear dispersive estimate \eqref{dispersive} and Hölder's inequality, we may estimate

\[
\Big\| \int_{0}^{t} e^{i(t-s)\Delta} \big[ |u|^{3} u \big](s) \, ds \Big\|_{L_x^p}
\lesssim \int_{0}^{t} |t-s|^{-3\bigl( \frac{1}{2} - \frac{1}{p} \bigr)}
\| u(s) \|_{L^p} \, \| u(s) \|_{L_x^{\frac{3p}{p-2}}}^{3} \, ds.
\]
From the earlier definition, \( |s|^{3\bigl( \frac{1}{2} - \frac{1}{p} \bigr)} \| u(s) \|_{L^p} \le \| u \|_{X(s)} \). Hence,
\begin{equation}\label{20}
    \Big\| \int_{0}^{t} e^{i(t-s)\Delta} \big[ |u|^{3} u \big](s) \, ds \Big\|_{L_x^p}
\lesssim \int_{0}^{t} |t-s|^{-3\bigl( \frac{1}{2} - \frac{1}{p} \bigr)}
|s|^{-3\bigl( \frac{1}{2} - \frac{1}{p} \bigr)}
\| u \|_{X(s)} \, \| u(s) \|_{L_x^{\frac{3p}{p-2}}}^{3} \, ds.
\end{equation} 

Then we decompose $[0,t)$ into $[0,\frac{t}{2})$ and $[\frac{t}{2},t)$. For $s\in[0,\frac{t}{2})$, we note that $|t-s|\approx|t|$, and for $s\in [\frac{t}{2},t)$, we note that $|s|\approx|t|$. So we find
\begin{align*}
\left\| \int_{0}^{t} e^{i(t-s) \Delta} \left[ |u|^{3} u \right] (s) \, ds \right\|_{L_x^p}
&\lesssim |t|^{-3\left( \frac{1}{2} - \frac{1}{p} \right)} 
\int_{0}^{t/2} |s|^{-3\left( \frac{1}{2} - \frac{1}{p} \right)} 
\|u\|_{X(s)} \|u(s)\|^3_{L_x^{\frac{3p}{p-2}}} \, ds \\
&\quad + |t|^{-3\left( \frac{1}{2} - \frac{1}{p} \right)} 
\int_{t/2}^{t} |t - s|^{-3\left( \frac{1}{2} - \frac{1}{p} \right)} 
\|u\|_{X(s)} \|u(s)\|^3_{L_x^{\frac{3p}{p-2}}} \, ds.
\end{align*}
Since \( 2 < p \le 6 \), both \( |s|^{-3(\frac{1}{2} - \frac{1}{p})} \) and \( |t-s|^{-3(\frac{1}{2} - \frac{1}{p})} \) belong to the Lorentz space \( L_s^{\frac{2p}{3p-6}, \infty} \). By Hölder's inequality, we obtain
\[
\left\| \int_{0}^{t} e^{i(t-s)\Delta} \left[ |u|^{3} u \right] (s) \, ds \right\|_{L_x^p} 
\lesssim |t|^{-3(\frac{1}{2}-\frac{1}{p})}\big\| 
\| u \|_{X(s)} \| u(s) \|_{L_x^{\frac{3p}{p-2}}}^{3}\big\|_{L_s^{\frac{2p}{6-p},1}([0,t))} 
\left\| s^{-3(\frac{1}{2}-\frac{1}{p})} \right\|_{L_S^{\frac{2p}{3(p-2)},\infty}}.
\]
\[\lesssim |t|^{-3(\frac{1}{2}-\frac{1}{p})}\big\| 
\| u \|_{X(s)} \| u(s) \|_{L_x^{\frac{3p}{p-2}}}^{3}\big\|_{L_s^{\frac{2p}{6-p},1}([0,t))}
\]
For $t \in [0, T_j)$, we now decompose $[0,t)$ into $[0,t)\cap[0,T_{j-1})$ and $[0,t)\cap I_j$. Doing so, \eqref{boot} and proposition \ref{Spacetime_bounds} then imply
\begin{align*}
&\big\| 
\| u \|_{X(s)} \| u(s) \|_{L_x^{\frac{3p}{p-2}}}^{3}\big\|_{L_s^{\frac{2p}{6-p},1}([0,t))} \\
&\leq \|u\|_{X(T_{j-1})} \|u\|_{L_s^{\frac{6p}{6-p},3} L_x^{\frac{3p}{p-2}}([0,T_{j-1}))}^3 + \|u\|_{X(T_j)} \|u\|_{L_s^{\frac{6p}{6-p},3} L_x^{\frac{3p}{p-2}}(I_j)}^3 \\
&\leq C(\|u_0\|_{\dot{H}^{\frac{5}{6}}})\|u\|_{X(T_{j-1})} + \eta^{3}\|u\|_{X(T_j)}.
\end{align*}
Along with the linear term (\ref{2}), this yields the bootstrap statement (\ref{1}) and concludes the proof of the integrable case of Theorem \ref{main thm}.
\end{proof}
\subsection{Non-integrable case} 
Now we consider the case \( p > 6 \). We follow the structure of the proof of the integrable case, taking care now to avoid the non-integrability of the linear dispersive decay \eqref{dispersive} near \( t = 0 \).

Following the notation in \cite{Dispersivedecay}, in the non-integrable case, for each \( t > 0 \) we will decompose the integral over \( [0, t) \) into an early-time interval \( [0, \frac{t}{2}) \) and a late-time interval \( [\frac{t}{2}, t) \). Unlike the integrable case, these intervals must be treated separately.

We first consider the early-time interval \( [0, \frac{t}{2}) \). On this interval, we carefully apply the integrable case of Theorem \ref{main thm} to produce a factor of \( \| u_0 \|_{L^{p'}} \). The purpose of this is to avoid generating a non-integrable term \( |s|^{-3(\frac{1}{2} - \frac{1}{p})} \), as seen in equation (\ref{20}). Since this argument is independent of the bootstrap structure, we present this estimate in Lemma \ref{4}.
\begin{lemma} \label{4}
    (Early-time interval). Fix $6<p<\infty$. Suppose that $u_0\in H^{\frac{5}{6}}\cap L^{p'}(\mathbb{R}^3)$ satisfies the hypotheses of Theorem \ref{well-posedness}. Then the corresponding solution u(t) to (NLS) with initial data $u_0$ satisfies
    \[
\left\|\int_{0}^{t/2} e^{i(t-s)\Delta} \big[|u|^3 u\big](s) \, ds\right\|_{L^p} 
\leq C(\|u_0\|_{\dot{H}^{\frac{5}{6}}}) |t|^{-3(\frac{1}{2}-\frac{1}{p})} \|u_0\|_{L^{p'}}.
\]
\end{lemma}
To prove this lemma, we need the following lemma concerning estimates for the linear term.
\begin{lemma} \label{5}
    Fix \( 1 \le q \le 2 \) and suppose that \( f \in L^q \cap H^{\frac{5}{6}}(\mathbb{R}^3) \). Then
    \[
    \| e^{it\Delta} f \|_{L_t^4 L_x^{4q}} \lesssim \| f \|_{L^q}^{1/4} \| f \|_{H^{\frac{5}{6}}}^{3/4}.
    \]
\end{lemma}

\begin{proof}
    Consider a Littlewood--Paley piece \( f_N \) for some \( N \in 2^{\mathbb{Z}} \). Since \( (4, 3) \) is a Schr\"{o}dinger admissible pair, we may apply Strichartz estimates and Bernstein's inequality \ref{Bernstein_inequality} to obtain
    \[
    \bigl\| e^{it\Delta} f_N \bigr\|_{L_t^4 L_x^{4q}} 
    \lesssim N^{\frac{4q-3}{4q}} \bigl\| e^{it\Delta} f_N \bigr\|_{L_t^4 L_x^3}
    \lesssim N^{\frac{4q-3}{4q}} \| f_N \|_{L^2}.
    \]
    Applying Bernstein's inequality \ref{Bernstein_inequality} once more, we obtain two distinct estimates for \( e^{it\Delta} f_N \).
\begin{equation} \label{25}
    \left\| e^{it\Delta} f_N \right\|_{L_t^4 L_x^{4q}} \lesssim N^{\frac{2q-9}{12q}} \left\| \nabla^{\frac{5}{6}} f_N \right\|_{L^2},
\end{equation}
\begin{equation} \label{26}
    \left\| e^{it\Delta} f_N \right\|_{L_t^4 L_x^{4q}} \lesssim N^{\frac{9-2q}{4q}} \left\| f_N \right\|_{L^q}. 
\end{equation}
We now decompose \( f \) into high-frequency and low-frequency parts based on a frequency cutoff parameter \( M \in 2^{\mathbb{Z}} \). We apply \eqref{25} to the high-frequency part and \eqref{26} to the low-frequency part. Consequently, we obtain

\begin{align*}
\bigl\| e^{it\Delta} f \bigr\|_{L_t^4 L_x^{4q}} 
&\lesssim \sum_{N > M} N^{\frac{2q-9}{12q}} \| f_N \|_{{H}^{\frac{5}{6}}} 
   + \sum_{N \le M} N^{\frac{9-2q}{4q}} \| f_N \|_{L^q} \\
&\lesssim M^{\frac{2q-9}{12q}} \| f \|_{{H}^{\frac{5}{6}}} 
   + M^{\frac{9-2q}{4q}} \| f \|_{L^q}.
\end{align*}

Choosing
\[
M^{\frac{9-2q}{3q}} \approx \frac{\| f \|_{{H}^{\frac{5}{6}}}}{\| f \|_{L^q}},
\]
we thus complete the proof of the lemma.
\end{proof}
To prove Lemma \ref{4}, we similarly follow the approach in \cite{Dispersivedecay} by decomposing \( u \) into a linear term and a nonlinear correction:
\begin{equation}
    u(t) = e^{it\Delta} u_0 \mp i \int_{0}^{t} e^{i(t-s)\Delta} \bigl[ |u|^{3} u \bigr](s) \, ds 
         = e^{it\Delta} u_0 + u_{\text{nl}}.
\end{equation}
We may now proceed with the proof of Lemma \ref{4}.
\begin{proof}
    As previously noted, we observe that \( |t-s| \approx |t| \) for \( s \in [0, \frac{t}{2}) \). Then by the linear dispersive decay,

\[
\Bigl\| \int_{0}^{t/2} e^{i(t-s)\Delta} \bigl[ |u|^{3} u \bigr](s) \, ds \Bigr\|_{L^p} 
\lesssim |t|^{-3\bigl( \frac{1}{2} - \frac{1}{p} \bigr)} 
\int_{0}^{t/2} \bigl\| \bigl[ |u|^{3} u \bigr](s) \bigr\|_{L^{p'}} \, ds.
\]

Substituting the decomposition of \( u \) and applying Lemma \ref{5}, we obtain
    \begin{equation} \label{9}
        \begin{aligned}
\left\| \int_{0}^{t/2} e^{i(t-s)\Delta} \left[ |u|^3 u \right](s) ds \right\|_{L^p}
&\lesssim |t|^{-3\left( \frac{1}{2} - \frac{1}{p} \right)} \left[ \left\| e^{is\Delta} u_0 \right\|_{L_s^4 L_x^{4p'}}^3 + \sum_{\alpha=0}^3 \int_{0}^{t/2} \left\| \left( e^{is\Delta} u_0 \right)^\alpha u_{nl}^{4-\alpha} \right\|_{L^{\frac{p}{p-1}}} ds \right] \\
&\leq C(\| u_0 \|_{\dot{H}^{\frac{5}{6}}}) |t|^{-3\left( \frac{1}{2} - \frac{1}{p} \right)} \left[ \| u_0 \|_{L^{p'}} + \sum_{\alpha=0}^2 \int_{0}^{t/2} \left\| \left( e^{is\Delta} u_0 \right)^\alpha u_{nl}^{4-\alpha} \right\|_{L^{\frac{p}{p-1}}} ds \right] \\
&= C(\| u_0 \|_{\dot{H}^{\frac{5}{6}}}) |t|^{-3\left( \frac{1}{2} - \frac{1}{p} \right)} \left[ \| u_0 \|_{L^{p'}} + \sum_{\alpha=0}^3 I_\alpha \right].
\end{aligned}
    \end{equation}
For notational convenience, we define \( \beta = \frac{29p-18}{14p} \). By Hölder's inequality, we may estimate \( I_{\alpha} \) as

\begin{equation} \label{7}
I_\alpha \lesssim \Bigl\| \, 
\bigl\| e^{is\Delta} u_0 \bigr\|_{L_x^{\frac{87p-54}{35p}}}^{\alpha \vee \beta} 
\bigl\| u_{\text{nl}}(s) \bigr\|_{L_x^{\frac{87p-54}{35p}}}^{\beta - \alpha \vee \beta} 
\bigl\| \bigl( e^{is\Delta} u_0 \bigr)^{\alpha - \alpha \wedge \beta} 
u_{\text{nl}}^{3 - \alpha \wedge \beta}(s) \bigr\|_{L_x^{\frac{6p}{p-6}}} 
\Bigr\|_{L_s^1},
\end{equation}

where the nonlinear term \( u_{\text{nl}} \) satisfies
\begin{equation} \label{6}
    \| u_{nl}(s) \|_{L_x^{\frac{87p-54}{35p}}} \leq \| e^{is\triangle} u_0 \|_{L_x^{\frac{87p-54}{35p}}} + \| u(s) \|_{L_x^{\frac{87p-54}{35p}}} \leq C(\| u_0 \|_{\dot{H}^{\frac{5}{6}}}) |s|^{-3(\frac{1}{2}-\frac{35p}{87p-54})} \| u_0 \|_{L^{\frac{87p-54}{52p-54}}}.
\end{equation}
A direct computation combined with the Sobolev embedding yields

\begin{equation}\label{8}
    \| u_0 \|_{L_x^{\frac{87p-54}{52p-54}}}^\beta 
    \lesssim \| u_0 \|_{L_x^{p'}} \, \| u_0 \|_{L_x^{\frac{9}{2}}}^{\beta-1} 
    \leq C\bigl( \| u_0 \|_{\dot{H}^{\frac{5}{6}}} \bigr) \, \| u_0 \|_{L_x^{p'}}.
\end{equation}
Substituting \eqref{6} and \eqref{8} into \eqref{7}, we obtain
\[
\begin{aligned}
I_\alpha &\lesssim \left\| |s|^{-\frac{17p-54}{28p}} \| u_0 \|_{L_x^{\frac{87p-54}{52p-54}}}^\beta \left\| \left( e^{is\Delta} u_0 \right)^{\alpha - \alpha \vee \beta} u_{nl}^{3 - \alpha \wedge \beta}(s) \right\|_{L_x^{\frac{6p}{p-6}}} \right\|_{L_s^1} \\
&\lesssim C(\| u_0 \|_{\dot{H}^{\frac{5}{6}}}) \| u_0 \|_{L_x^{p'}} \left\| |s|^{-\frac{17p-54}{28p}} \left\| \left( e^{is\Delta} u_0 \right)^{\alpha - \alpha \vee \beta} u_{nl}^{3 - \alpha \wedge \beta}(s) \right\|_{L_x^{\frac{6p}{p-6}}} \right\|_{L_s^1}.
\end{aligned}
\]
Since for \( p > 6 \) we have \( 0 < \frac{17p-54}{28p} < 1 \), it follows that \( |s|^{-\frac{17p-54}{28p}} \in L^{\frac{28p}{17p-54}, \infty} \). Applying Hölder's inequality, we obtain

\[
I_{\alpha} \lesssim C\bigl( \| u_{0} \|_{\dot{H}^{\frac{5}{6}}} \bigr) \,
\| u_{0} \|_{L_{x}^{p'}} \,
\Bigl\| \bigl( e^{i s \Delta} u_{0} \bigr)^{\alpha - \alpha \vee \beta} 
u_{\text{nl}}^{3 - \alpha \wedge \beta}(s) \Bigr\|_{L_{s}^{\frac{28p}{11p+54}, 1} L_{x}^{\frac{6p}{p-6}}} .
\]

Noting that \( \alpha - \alpha \vee \beta + 4 - \alpha \wedge \beta = 4 - \beta \), and applying Hölder's inequality once again, we derive the estimate
\[
\begin{aligned}
I_{\alpha} &\lesssim C\left(\|u_{0}\|_{\dot{H}^{\frac{5}{6}}}\right)\|u_{0}\|_{L_{x}^{p'}}\left\|e^{is\Delta}u_{0}\right\|_{L_{s}^{\frac{28p(4-\beta)}{11p+54},\infty}L_{x}^{\frac{6p(4-\beta)}{p-6}}}^{\alpha-\alpha\vee\beta}\left\|u_{nl}\right\|_{L_{s}^{\frac{28p(4-\beta)}{11p+54},4-\alpha\wedge\beta}L_{x}^{\frac{6p(4-\beta)}{p-6}}}^{4-\alpha\wedge\beta} \\
&= C\left(\|u_{0}\|_{\dot{H}^{\frac{5}{6}}}\right)\|u_{0}\|_{L_{x}^{p'}}\left\|e^{is\Delta}u_{0}\right\|_{L_{s}^{\frac{54p+36}{11p+54},\infty}L_{x}^{\frac{81p+54}{7p-42}}}^{\alpha-\alpha\vee\beta}\left\|u_{nl}\right\|_{L_{s}^{\frac{54p+36}{11p+54},4-\alpha\wedge\beta}L_{x}^{\frac{81p+54}{7p-42}}}^{4-\alpha\wedge\beta}.
\end{aligned}
\]
A direct computation shows that \( \left( \frac{54p+36}{11p+54}, \frac{81p+54}{7p-42} \right) \) is Schr\"{o}dinger-admissible with \( \frac{5}{6} \) spatial derivatives, and that

\begin{equation*}
    4 - \alpha \wedge \beta \ge 4 - \beta = \frac{27p + 18}{14p} \ge \frac{4}{5}
\end{equation*}
holds for \( p > 6 \). Using Corollary \ref{Nonlinear_correction_bounds} and Proposition \ref{Spacetime_bounds}, we obtain
\[
I_{\alpha} \le C\bigl( \|u_{0}\|_{\dot{H}^{\frac{5}{6}}} \bigr) \, \|u_{0}\|_{L^{p'}},
\]
for \( \alpha = 0, 1, 2, 3\). Together with \eqref{9}, this completes the proof of Lemma \ref{4}.
\end{proof}
We now turn to the proof for the late-time interval \( [\frac{t}{2}, t) \). On this interval, we employ a Sobolev embedding before applying the linear dispersive decay \eqref{dispersive}. This reduces the Lebesgue exponent below the integrability threshold 6. This part of the proof will be incorporated directly into the proof of Theorem \ref{main thm}.

\noindent \textbf{Proof of Theorem \ref{main thm}}. Next, we consider the case \( 6 < p < \infty \). As before, it suffices to treat \( t > 0 \). By the density of Schwartz functions in \( \dot{H}^{\frac{5}{6}} \cap L^{p'} \), we may restrict attention to Schwartz solutions of (NLS). By proposition \ref{Spacetime_bounds}, we decompose $[0,\infty)$ into $J=J(\eta,\|u_0\|_{\dot{H}^{\frac{5}{6}}})$ many intervals $I_j=[T_{j-1},T_j)$ on which 
\begin{equation}\label{boot2}
    \| \nabla^{\frac{5}{6}} u \|_{L_t^{\frac{6p}{p+6},3} L_x^{\frac{18p}{7p-12}} (I_j)} < \eta.
\end{equation}
Again, we only need to prove \eqref{1}. Combining the linear dispersive decay \eqref{dispersive} with Lemma \ref{4}, the Duhamel formula implies that for all \( j \),

\begin{align} \label{12}
\|u\|_{X(T_j)} 
&\lesssim C\bigl( \|u_{0}\|_{\dot{H}^{\frac{5}{6}}} \bigr) \|u_{0}\|_{L^{p'}} 
+ \Bigl\| \int_{\frac{t}{2}}^{t} e^{i(t-s)\Delta} \bigl[ |u|^{3} u \bigr](s) \, ds \Bigr\|_{X(T_j)} .
\end{align}
Thus it remains only to treat the late-time interval \( [t/2, t) \).
Applying the Sobolev embedding together with the linear dispersive decay \eqref{dispersive}, we find that
\begin{align}
\left\| \int_{\frac{t}{2}}^{t} e^{i(t-s)\Delta} \left[ |u|^{3} u \right](s) ds \right\|_{L^p}
&\lesssim \int_{\frac{t}{2}}^{t} \left\| e^{i(t-s)\Delta} |\nabla| \left[ |u|^{3} u \right](s) \right\|_{L_x^{\frac{3p}{p+3}}} ds \\
&\lesssim \int_{\frac{t}{2}}^{t} |t - s|^{-\frac{p-6}{2p}} \left\| |\nabla| \left[ |u|^{3} u \right](s) \right\|_{L_x^{\frac{3p}{2p-3}}} ds.
\end{align}
Applying the Gagliardo--Nirenberg inequality and Hölder's inequality, we obtain
\begin{align*}
\Bigl\| \int_{\frac{t}{2}}^{t} e^{i(t-s)\Delta} \bigl[ |u|^{3} u \bigr](s) \, ds \Bigr\|_{L^p}
&\lesssim \int_{\frac{t}{2}}^{t} |t - s|^{-\frac{p-6}{2p}} 
\|u(s)\|_{L_x^p} \, 
\bigl\| \nabla^{\frac{5}{6}} u(s) \bigr\|_{L_x^{\frac{18p}{7p-12}}}^{3} \, ds.
\end{align*}

We note that \( |s| \approx |t| \) for \( s \in [t/2, t) \). Using Hölder's inequality, we then obtain
\begin{equation} \label{10}
\begin{aligned}
\left\| \int_{\frac{t}{2}}^{t} e^{i(t-s)\Delta} \left[ |u|^{3} u \right](s) ds \right\|_{L^p}
&\lesssim |t|^{-3\left(\frac{1}{2} - \frac{1}{p}\right)} \int_{\frac{t}{2}}^{t} |t - s|^{-\frac{p-6}{2p}} \|u\|_{X(s)} \|\nabla^{\frac{5}{6}} u(s)\|_{L_x^{\frac{18p}{7p-12}}}^{3} ds \\
&\approx|t|^{-3\left(\frac{1}{2} - \frac{1}{p}\right)} \left\| \|u(s)\|_{X(s)} \|\nabla^{\frac{5}{6}} u(s)\|_{L_x^{\frac{18p}{7p-12}}}^{3} \right\|_{L_s^{\frac{2p}{p+6},1} \left( [0,t) \right)}.
\end{aligned}
\end{equation}
Taking the supremum over \( t \in [0, T_j) \) and decomposing \( [t/2, t) \) into \( [t/2, t) \cap I_j \) and \( [t/2, t) \cap [0, T_{j-1}) \), we obtain the following. Since \( \bigl( \frac{6p}{p+6}, \frac{18p}{7p-12} \bigr) \) is a non-endpoint Schr\"{o}dinger-admissible pair, Propositions \ref{Spacetime_bounds} and \eqref{boot2} imply

\begin{equation}\label{11}
\begin{aligned}
\Bigl\| \int_{\frac{t}{2}}^{t} e^{i(t-s)\Delta} \bigl[ |u|^{3} u \bigr](s) \, ds \Bigr\|_{X(T_j)}
&\lesssim \|u\|_{X(T_{j-1})} \, 
\bigl\| \nabla^{\frac{5}{6}} u \bigr\|_{L_t^{\frac{6p}{p+6},3} L_x^{\frac{18p}{7p-12}}}^{3} 
+ \|u\|_{X(T_j)} \, 
\bigl\| \nabla^{\frac{5}{6}} u \bigr\|_{L_t^{\frac{6p}{p+6},3} L_x^{\frac{18p}{7p-12}}(I_j)}^{3} \\
&\lesssim C\bigl( \|u_0\|_{\dot{H}^{\frac{5}{6}}} \bigr) 
\bigl( \|u\|_{X(T_{j-1})} + \eta^{3} \|u\|_{X(T_j)} \bigr).
\end{aligned}
\end{equation}

Combining \eqref{11} with \eqref{12} yields the bootstrap statement \eqref{1}. Together with the earlier considerations, choosing \( \eta \) sufficiently small and iterating over \( j = 1, \dots, J\bigl( \|u_0\|_{\dot{H}^{\frac{5}{6}}} \bigr) \) completes the proof of Theorem \ref{main thm}.

\bigskip
\noindent {\bf Acknowledgements.}





\end{document}